\RequirePackage{fix-cm}
\documentclass[smallextended]{svjour3}       % onecolumn (second format)
\smartqed  % flush right qed marks, e.g. at end of proof
\usepackage{graphicx}

\usepackage[centertags]{amsmath}
\usepackage{amssymb}
\usepackage{ subfigure}
\usepackage{pinlabel}

\usepackage{color}
\usepackage{fullpage}
\usepackage{tikz}
\usetikzlibrary{shapes,arrows,calc}
\usetikzlibrary{decorations.markings}
\newtheorem{observation}{Observation}
 \journalname{Natural Computing}
\begin{document}

\title{DNA origami and the complexity of Eulerian circuits with turning costs\thanks{The work of the first and fourth authors was supported by the National Science Foundation (NSF) under grants  DMS-1001408 and EFRI-1332411.}}



\author{Joanna A. Ellis-Monaghan \and Andrew McDowell  \mbox{ \and Iain Moffatt}         \and
        Greta Pangborn %etc.
}


\institute{J. Ellis-Monaghan  \at
              Department of Mathematics, Saint Michael's College, One Winooski Park, Colchester, VT, 05439, USA \\
              %Tel.: 802-654-2660\\
              %Fax: 802 654 2610\\
              \email{jellis-monaghan@smcvt.edu}        %  \\
 \and
           A. McDowell \at
              	Department of Mathematics, Royal Holloway, University of London, Egham, Surrey, TW20 0EX, United Kingdom \\
             % Tel.: 01784 414686\\
             % Fax: 01784 430766\\
              \email{Andrew.Mcdowell.2010@live.rhul.ac.uk}
           \and
           I. Moffatt \at
              	Department of Mathematics, Royal Holloway, University of London, Egham, Surrey, TW20 0EX, United Kingdom \\
             % Tel.: 01784 414686\\
             % Fax: 01784 430766\\
              \email{iain.moffatt@rhul.ac.uk}
 \and
           G. Pangborn \at
               Department of Computer Science, Saint Michael's College, One Winooski Park, Colchester, VT, 05439, USA \\
              %Tel.: 802-654-2791\\
             % Fax: 802 654 2610\\
             \email{gpangborn@smcvt.edu}
}

\date{Received: date / Accepted: date}

\maketitle
%\linenumbers

\begin{abstract}
Building a structure using self-assembly of DNA molecules by origami folding requires finding a route for the scaffolding strand through the desired structure.  When the target structure is a 1-complex (or the geometric realization of a graph), an optimal route corresponds to an Eulerian circuit through the graph with minimum turning cost.  By showing that it leads to a solution to the 3-SAT problem, we prove that the general problem of finding an optimal route for a scaffolding strand for such  structures is NP-hard. We then show that the problem may readily be transformed into a Traveling Salesman Problem (TSP), so that  machinery that has been developed for the TSP may  be applied to find optimal routes for the scaffolding strand in a DNA origami self-assembly process. We give results for a few special cases, showing for example that the problem remains intractable for graphs with maximum degree 8, but is polynomial time for 4-regular plane graphs if the circuit is restricted to following faces. We conclude with some implications of these results for related problems, such as biomolecular computing and mill routing problems.


\end{abstract}

%\thanks{}

\section{Introduction}
Given an Eulerian graph (a connected graph in which all the vertices have even degree), it is well known that an Eulerian circuit, that is, a circuit that traverses each edge exactly once, can be found in polynomial time.  Here we have the additional information of a cost associated with each of the possible routes that a circuit can take through a vertex (the turnings), and we seek a lowest cost Eulerian circuit.  We show that finding an Eulerian circuit using a  best possible set of turnings is in general NP-hard.  This question arose as a design strategy problem for DNA self-assembly via origami folding, which involves finding an optimal route for the scaffolding strand of DNA through the targeted structure.

Self-assembly is the physical process by which structures form from disordered components, without outside direction, based on the local chemical and physical properties of the materials used. DNA naturally possesses properties conducive to self-assembly since DNA strands bond through base pairs within the strands following well-understood rules. For DNA self-assembly, DNA strands are designed so that when base pairing occurs, a molecule forms with the desired geometric structure.  Several methods of DNA self-assembly have been implemented, such as those using branched junction molecules pioneered by Seeman, \cite {CS91}, and those using DNA origami methods pioneered by Rothmund, \cite{Rot06}. A wide range of possible applications have been proposed for DNA self-assembly such as nanoscale circuitry and robotics, drug delivery systems, and biomolecular computing. DNA self-assembly methods and their applications are surveyed in, for example, \cite{Luo03},  \cite{PH11}, and \cite{San10}.

In DNA origami methods of nanoscale self-assembly, a single scaffolding strand of DNA traces the construct exactly once, and then short helper strands, called staples, bond to this strand to fold and lock it into the desired configuration (see, for example, \cite{HLS09}, \cite{NHLY10}, and  \cite{Rot06}). The design process for DNA origami assembly involves finding a route for a scaffolding strand through the desired structure.  While originally applied to 2-complexes (solid 2D) structures, and later to 3-complexes (solid 3D structures), that are `filled' by the strands of DNA, a logical next step is adapting this technique to 1-complexes, or graph-theoretical structures, such the skeletons of polyhedra.   Such graph-theoretical structures (cubes \cite{CS91}; truncated octahedra \cite{ZS94}; rigid octahedra \cite{SQJ04};  tetrahedra, dodecahedra, and buckyballs \cite{H+08}; and a 3D crystalline lattice \cite{Z+09}) have already been assembled via branched junction molecules.  It is  now reasonable to try to assemble these and similar structures from DNA origami.  However, the design strategies for `filled' constructions, such as the stars and smiley faces of \cite{Rot06}, or the 3D solid bricks, honeycombs and modularly assembled icosahedra of \cite{DDS09}, are different from those needed for open, graph-theoretical structures such as 1-complexes.  If the structure is a 1-complex or graph embedded in 3-space, for example a polyhedral skeleton, then, since the scaffolding strand is usually a single circular strand of DNA, its route must correspond to an Eulerian circuit through the graph or through some augmentation of the graph (if it is not Eulerian, for example).  We focus here on these structures that require an Eulerian circuit as the route for the scaffolding strand.

In general, since DNA bonding of complementary base pairs is energetically favourable, a system will tend to maximise the number of matches naturally, according to the laws of thermodynamics. However, other physical properties and behaviours of DNA strands may influence the shapes it forms.  Thus, in a DNA origami construction of a 1-complex there may be preferred ways for the scaffolding strand to pass through each vertex, for example, following a face of the structure rather than weaving through the vertex.  This leads to the associated the graph theoretic problem of finding an Eulerian circuit with minimum turning costs.  For example, a turning in an Euler circuit would have low cost if it corresponds to a configuration the scaffolding strand would readily adopt, a medium cost if the strand can be made to conform to the configuration albeit perhaps with some difficulty, and a high cost if the DNA strand is physically constrained from the configuration.  Ideally, the route for the scaffolding strand would only require turns that the strand follows readily, i.e., those that correspond to the turnings in an Euler circuit with minimum turning costs.

We show below that finding an Eulerian circuit with minimum turning cost is in general NP-hard by proving that it implies a solution to the 3-SAT problem, which is well-known to be NP-hard. This result has significant ramifications for using DNA origami as a basis for biomolecular computing of graph invariants. (Graph invariants are properties of graphs that remain unchanged under isomorphism.) Many graph invariants (for example, the existence of a Hamilton cycle, graph colorability, etc.) are known to be NP-hard. However, biomolecular computing strategies have been proposed for them (see, for example, \cite{Adl94} for Hamilton cycles, and \cite{JK99} for 3-SAT and vertex 3-colorability).  In order to compute a graph invariant via a biological process, the graph must first be encoded in molecular structures.  The results here show that assembling a graph from DNA origami as a first step in the computation must be approached with caution, as finding a good self-assembly strategy may be a priori intractable.

However, there is good news from a more pragmatic view point.  We also show that the minimum cost Eulerian circuit problem may be transformed in polynomial time into a Traveling Salesman Problem (TSP).  While the TSP is also in general NP-hard, extensive work has been done on this problem (see \cite {LL85} for a comprehensive overview), and the results here mean that  machinery developed for the TSP may now be brought to bear on finding optimal routes for a scaffolding strand for DNA origami assembly of reasonably-sized graph-theoretical structures for practical applications.

We also give the complexity of some special cases of the turning cost problem, showing for example that if the Eulerian graph is 8-regular, then the problem remains intractable, but that there is a polynomial time algorithm for certain classes of 4-regular graphs.  We also discuss some implications of the results here for biomolecular computing, and for the mill routing problem as given in \cite{A+05}.

\section{Graph theoretical background and problem statement}

The following conventions are used throughout this paper.  Further details and a full formalization of these concepts, including the transitions discussed below, may be found, for example, in~\cite{Fle90,Fle91}.

 Graphs are finite and may have loops and multiple edges.  Thus, a graph $G$ consists of a finite set of {\em vertices} denoted $V(G)$, and a finite multiset of {\em edges}, denoted $E(G)$, that are unordered pairs $(u,v)$ of vertices, with $v=u$ in the case of a loop. We generally use $n$ to denote $|V(G)|$  and call it the {\em size} of $G$. As usual, indices may be used as needed to distinguish among multiple edges: if $(u,v)$ has multiplicity $m$ then we index the $m$ copies with $1, 2, \ldots ,m$.  However, following standard convention, we will typically suppress the index and  just write $(u,v)$ for an edge and refer to the edge multiset $E(G)$ simply as the edge set.  Equivalently, a multigraph may be defined as a triple $(V,E, f)$, where $V$ and $E$ are disjoint sets of vertices and edges respectively, and $f$ is a function from $E$ to the set of unordered pairs of vertices that specifies the endpoints of each edge.  In either case, all edges are distinguishable.

A graph is \emph{planar} if it may be drawn in the plane without any edges crossing.   A \emph{plane graph} is a planar graph drawn in the plane.

  Intuitively, if an edge is thought of as a line segment between two vertices $u$ and $v$, and $p$ is the midpoint, then the two half-edges are the line segments $up$ and $vp$. If $e=(u,v)$ is an edge, then the half-edge incident with $u$ is formally denoted by $(u, e)$.  The two half-edges incident with a loop may be arbitrarily assigned indices to distinguish them if necessary, but as with multiple edges, when there is no danger of confusion we typically suppress the index. Note that, using indices as needed for loops and multiple edges, all half-edges are distinguishable, and the edge set of a graph is uniquely determined by its half-edges.  The {\em degree},  $d(v)$,  of a vertex $v$ is the number of half-edges incident with it. The {\em maximum degree} of a graph $G$ is $\Delta(G):= \max_{v\in V(G)} \{d(v)\} $.

The application we consider here involves Eulerian graphs, which are connected graphs wherein the degree of every vertex is even.  A {\em walk} traverses consecutive edges in a graph, allowing repeated edges and vertices; a {\em trail} allows repeated vertices but not edges; and a {\em path} repeats neither.  A {\em circuit} is a closed trail, and a {\em cycle} is a closed path.  Given a connected graph $G$, an {\em augmented graph} results from drawing an edge between any two vertices of odd degree, and continuing the process until no vertex of odd degree remains (a graph necessarily has an even number of odd degree vertices).  The resulting augmented graph is then Eulerian.

The DNA self-assembly application discussed in the introduction now motivates the following definition and problem formulation, with the turning costs corresponding to the prioritized set of preferred routes through each vertex for the scaffolding strand. We use the convention that the more preferable a turning, the lower its turning cost.

\begin{definition}[Turning cost.]\label{defn:turning costs}
Let $G$ be an Eulerian graph and $v$ be a vertex of $G$.
A {\em pairing} at $v$ is a set $\{(v,e), (v,f)\}$, where $(v,e)$ and $(v,f)$ are distinct half-edges incident with  $v$. To every pairing $\{(v,e), (v,f)\}$ we associate a non-negative rational number, called the {\em turning cost} of the pairing, denoted by $w_v(e,f)$. (The turning costs at $v$ can be thought of as a function $w_v$ from the set of all pairings at $v$ to the non-negative rational numbers.) When the half-edges involved in a turning is clear from the common vertex we will usually refer to the pairing and turning cost of the two edges rather than specifying the half-edges, e.g. simply say $\{e,f\}$ is a pairing at $v$. We call the set of costs at a vertex $v$  the {\em turning costs} at $v$.
\end{definition}

A {\em transition system}, $T(v)$, at $v$ is a set $S$ of pairings at $v$ such that every half-edge incident with  $v$ appears in exactly one pairing in $S$.  The cost of a transition system at $v$ is the sum of the turning costs over all pairings in the transition system, and is denoted $w(T(v))$.

Note that an Eulerian circuit $C$ determines a transition system at each vertex  by pairing half-edges at a vertex $v$ if they appear consecutively in $C$.  (The converse is not true: a set of transition systems at each vertex determines a disjoint set of circuits in a graph, but not necessarily an Eulerian circuit.) We denote the transition system at $v$ determined by an Eulerian circuit $C$ by $T_C(v)$.

If $C$ is an Eulerian circuit of a graph $G$, then the {\em cost}  of $C$, denoted $w(C)$, is the sum of the turning costs of all pairings that it determines:
\[  w(C)=\sum_{ v\in V(G)}   w(T_C(v)).   \]

The optimization problem arising from our DNA origami application may now be stated as follows:
\begin{problem}\label{prob:Eulerian Costs}
Given an Eulerian graph with turning costs, find an Eulerian circuit $C$ with minimum cost $w(C)$.
\end{problem}

We recall that, informally, a decision
problem is one for which there is a yes or no answer, such as, can graph $G$ be
colored using $k$ colors?  P is the set of decision problems for which it is possible to
determine the answer in polynomial time in the size of the input, and NP is
the set of decision problems for which it is possible to determine if a given answer is
correct in polynomial time in the size of the input.  Whether or not $P=NP$
remains a famous open question, but there is a large class of
problems, referred to as NP-hard, for which finding a polynomial time algorithm for any one of them
would automatically lead to polynomial time algorithms for all problems in NP.
A decision problem is NP-complete if it is both NP-hard and in NP.
See \cite{GJ79} for additional background.

Here, the decision problem corresponding to Problem~\ref{prob:Eulerian Costs} is the following:

\begin{problem}\label{prob:Eulerian decision}
Given an Eulerian graph $G$ equipped with a set of turning costs at each vertex, and a non-negative constant $c$, determine if there is an Eulerian circuit $C$ with the minimum cost $w(C) \leq c$.
\end{problem}

Note that if there were a polynomial time algorithm for Problem \ref{prob:Eulerian Costs}, then Problem \ref{prob:Eulerian decision} would be in P, since we could simply find a minimum cost Eulerian circuit and compare its cost to the given constant $c$.  However, we will show in the next section that these two problems are in general intractable by using the following special case of Problem \ref{prob:Eulerian decision}.

\begin{problem}\label{prob:zero one decision}
Given an Eulerian graph $G$ equipped with turning costs in $\{0,1\}$ at each vertex,  determine if there is a zero-cost Eulerian circuit $C$.
\end{problem}

\section{Finding an Eulerian circuit with minimum turning cost is NP-hard}

We will demonstrate that finding an Eulerian circuit with minimum turning cost is NP-hard by showing that 3-SAT is polynomial time reducible to this problem.  In particular, we note that if Problem \ref{prob:Eulerian decision} can be solved in polynomial time, then the special case Problem \ref{prob:zero one decision} can be solved in polynomial time.  However, we show that 3-SAT can be reformulated in polynomial time to the problem of finding such a zero-cost Eulerian circuit in an associated Eulerian graph with turning costs of zero or one.  This would imply that 3-SAT could be solved in polynomial time.  From this we conclude that Problem \ref{prob:Eulerian decision} is NP-complete, and hence Problem \ref{prob:Eulerian Costs} is NP-hard.

The 3-SAT problem involves a Boolean logic conjunctive normal form expression such as \[(x_{1}\vee \neg x_{2} \vee x_{3})\wedge (\neg x_{1}\vee x_{4} \vee x_{5})\wedge \dots \wedge (x_{2}\vee x_{3} \vee \neg x_{4}).\]  In a  3-SAT problem, each clause in such an expression has exactly three distinct variables from a set $\{x_1, x_2, \ldots , x_n \}$ of Boolean variables, that is, each variable may be assigned a value of true or false.  The symbols $\vee$, $\wedge$, $\neg$ represent the logic operations `and', `or', and `not', respectively.  A {\em literal} is a variable $x_i$ or its negation $\neg x_i$, with the former referred to as a {\em positive literal} and the latter as a {\em negative literal}. For the  3-SAT problem, only the logic operations $\lor$ and  $\neg$ may appear inside the clauses, and only $\land$ may join clauses. The decision problem asks whether there is an assignment of true or false to each of the variables such that the whole expression evaluates to true (see for example, \cite{KT05} pg. 459).

\subsection{Constructing an Eulerian graph with turning costs from a 3-SAT instance}  \label{sec:associated graph}

We begin with an arbitrary 3-SAT instance, and construct an associated Eulerian graph.
For expositional clarity, we  will describe this graph via an embedding in 3-space, but since the objective is an abstract Eulerian graph, the choice of specific embedding is irrelevant. The given embedding simply facilities our description of which pairs of half-edges should receive which turning costs.

For each Boolean variable $x_i$ in the 3-SAT expression, we draw a vertex $x_i$ in the plane.
We read the 3-SAT expression from left to right, and for each clause we create a triangle on the vertices $x_i$, $x_j$ and $x_k$, where $x_i,x_j,$ and $x_k$ are the variables in the clause. The edges of the triangle may cross in the plane, but we require each of the three pairings  of half-edges $\{(x_i, (x_i,x_j)),(x_i, (x_i,x_k))\}$,  $\{(x_j, (x_j,x_i)),(x_j, (x_j,x_k))\}$ and $\{(x_k, (x_k,x_i)),(x_k, (x_k,x_j))\}$ in the triangle to appear consecutively in the cyclic orders of the vertices at their common end points, i.e. the edges are neighbors in the plane drawing. Which of the two possible orders for two consecutive half-edges does not matter, provided they are consecutive.  The order of the pairs in the cyclic order about a vertex also does not matter.  At each vertex, we shade a small region between the half-edges in the same triangle to record this property, as in Figure \ref{f.tri1}.  Again we emphasize that the shading of regions and the embedding into 3-space are simply  expository conveniences for describing the construction; they are not necessary for the implementation.

\begin{figure}%[h]
\centering
\subfigure[ Triangles for 3-SAT formula: $(x_1\lor x_2\lor \lnot x_3)\land (\lnot x_1\lor x_3\lor x_4)\land ( x_1\lor x_2\lor \lnot x_5)$.]{
\labellist \small\hair 2pt
\pinlabel {$x_2$} at   37 4
\pinlabel {$x_5$} at   104 37
\pinlabel {$x_1$} at    98 126
\pinlabel {$x_3$} at   146 4
\pinlabel {$x_4$} at    204 100
\endlabellist
\includegraphics[scale=.8]{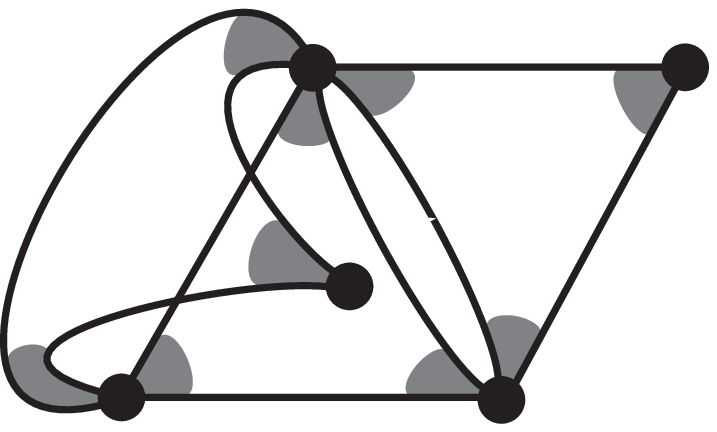}
 \label{f.tri1}
}
\hspace{10mm}
\subfigure[Initial neighbourhood of a vertex  $x_1$.]{
\labellist \small\hair 2pt
\pinlabel {$x_1$}  at  77 95
\pinlabel {$(x_1 \!  \! \lor  \! \! x_p  \! \! \lor   \! \!  x_q)$} [r] at   42 94
\pinlabel {$( x_1 \! \! \lor \! \!  x_l \! \! \lor \! \!   x_m)$} [l] at   114 94
\pinlabel {\rotatebox{90}{$( \lnot x_1 \! \! \lor \! \!  x_n \! \! \lor \! \!   x_r)$}}  at   78 -2
\endlabellist
\raisebox{10mm}{\includegraphics[scale=.8]{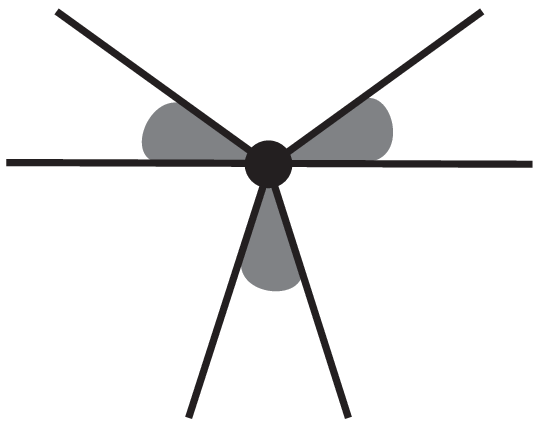}}
 \label{f.tri2}
}
\caption{Forming the triangles of the graph associated with a 3-SAT instance.}
\label{f.tri}
\end{figure}

We label each triangle with the clause it represents.  Observe that each vertex has even degree, with a neighbourhood consisting of alternating shaded and unshaded regions, where the shaded regions are parts of triangles labelled by clauses
containing the literal (either positive or negative) that labels the vertex.  See Figure~\ref{f.tri2}.

\begin{figure}
\centering
\subfigure[Adding the apex vertex $u$.]{
\labellist \small\hair 2pt
\pinlabel {$u$} at   208 195
\endlabellist
\raisebox{10mm}{\includegraphics[scale=.8]{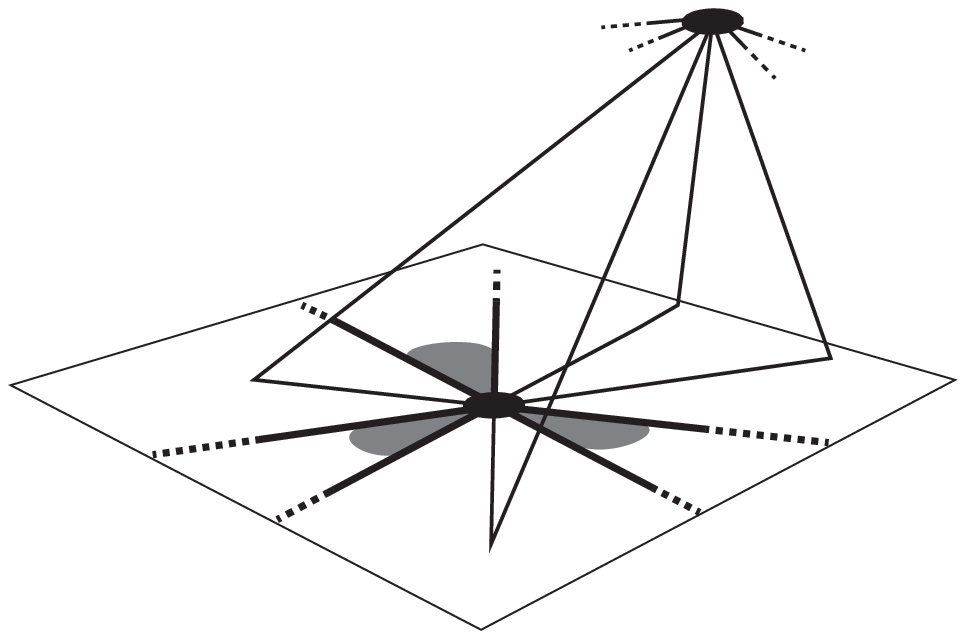}}
 \label{f.aug2}
}
\hspace{10mm}
\subfigure[Final neighbourhood of a vertex $x_1$.]{
\labellist \small\hair 2pt
\pinlabel {$(x_1 \! \! \lor \! \!  x_p \! \! \lor \! \!   x_q)$} [r] at   42 94
\pinlabel {$( x_1 \! \! \lor \! \!  x_l \! \! \lor \! \!   x_m)$} [l] at   114 94
\pinlabel {\rotatebox{90}{$( \lnot x_1 \! \! \lor  \! \! x_n \! \! \lor  \! \!  x_r)$}}  at   78 -2
\endlabellist
\raisebox{10mm}{\includegraphics[scale=.8]{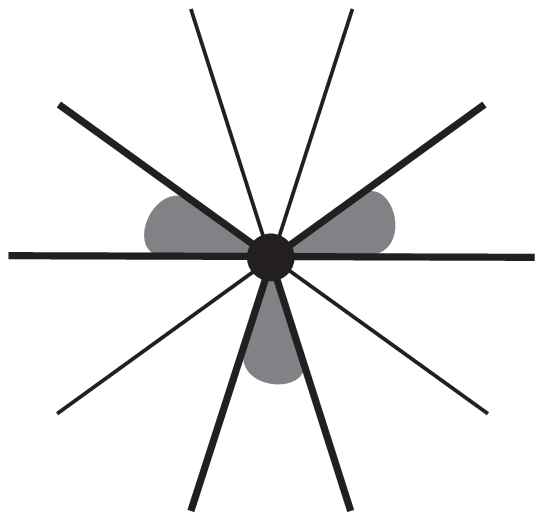}}
 \label{f.aug1}
}
\caption{Adding the apex vertex.}
\label{f.aug}
\end{figure}

We now add an extra \emph{apex} vertex, $u$, above the plane.  We will draw  edges (in general, there will be multiple edges) between $u$ and a vertex $x_i$ so that the half-edge incident with $x_i$ lies in the plane in a small neighbourhood of $x_i$ before rising above the plane to meet the other half of the edge  (the half-edge incident with  $u$). See Figure~\ref{f.aug2}. These half-edges incident with $x_i$ are drawn in the unshaded regions about $x_i$, according to the following scheme.  For each unshaded region at $x_i$, if the literal corresponding to $x_i$ is positive in both clauses labelling the shaded regions bounding the white region, or is negative in both, then we place two edges from $x_i$ to $u$ emerging from this unshaded region. If, however, the literal is positive in one clause, and negative in the other, we then place one edge emerging from the unshaded region, as in Figure~\ref{f.aug1}.

If we consider the sequence of literals corresponding to the variable $x_i$ in the triangles labelling the shaded regions about the vertex $x_i$, we notice that sequences of consecutive positive literals alternate with sequences of negative literals.  Each time there is a switch from positive to negative, or vice versa, we add a single edge, hence we have added an even number of single edges.  Each time there is no switch, we add two edges.  Thus, in total, we will have added an even number of edges, preserving the even degree of every vertex $x_i$, and ensuring that $u$ is also of even degree.
Since every vertex is connected to $u$, the graph is connected.  In addition every vertex has even degree, so the graph is Eulerian.

We now assign turning costs to the half-edge pairings on each vertex as follows. We assign a cost of zero to every pairing of half-edges incident with  the apex vertex $u$. We also assign a cost of zero to any consecutive pairing of half-edges in the clockwise orientation about each vertex $x_i$. All other pairings of half-edges have cost one.

Thus we have associated an Eulerian graph with turning costs to the given 3-SAT instance. For a 3-SAT instance $I$, we denote this graph by  $G_I$ and call it  the {\em Eulerian graph with turning costs of $I$}.

\subsection{The Eulerian graph $G_I$ may be constructed in polynomial time} \label{polytime graph}

\begin{proposition}\label{p.polyt}
Given a 3-SAT expression $I$ where $n$ is the number of distinct literals and $r$ is the number of clauses in the expression, then $G_I$, the Eulerian graph with turning costs of $I$, may be constructed in $O(n r^2)$ time.
  \end{proposition}

  \begin{proof}
 Since the vertices except $u$ are labeled by the literals, which are indexed from 1 to $n$, only $n$ needs to be known to create the vertex list, and $n$ may be found simply by recording the highest index appearing in the clauses as they are read one at a time.  Furthermore, each clause contributes three edges, the three half-edge pairings of which may added (in consecutive order) to the cyclic orders of the corresponding vertex.  Since the order of the pairs about the vertex does not matter, they may be added to the cyclic order as they are read in.  Thus, again, each clause need only be read once to create the list of these edges.

Adding the half-edges for the edges from each vertex $x_i$ to the apex vertex $u$ involves reading through the cyclic order about each vertex once.  Since each clause can contribute at most two half-edges incident with  a given vertex, there are at most $2r$ half-edges in the sequence for each vertex. Furthermore, we have to insert no more than $2r$ half-edges corresponding to edges from the vertex to $u$.  Thus, since there are $n$ vertices, this process takes $O(rn)$ steps.

Turning costs are assigned by reading  the list of half-edges at each vertex, and recording a zero for consecutive half-edges on the list, and a one for all other pairs.  Since no vertex may have degree greater than $4r$, there are at most ${4r}\choose {2}$ pairs of edges to assign turning costs for at each vertex except for $u$.   There is no need to record turning costs for pairs of edges incident with  $u$ as these are all zero and hence cannot contribute to the total cost of the Eulerian circuit.    Thus, listing all the turning costs requires $O(n r^2)$ steps.

Since listing the turning costs for each pair is the most time consuming step, the overall complexity of constructing the associated Eulerian graph with turning costs is  $O(n r^2)$.
\hfill$\qed$
\end{proof}

\subsection{Finding an Eulerian circuit with minimum turning costs is NP-hard}\label{polytime translate}

We use the construction of Subsection \ref{sec:associated graph} to show that the 3-SAT problem is polynomial time reducible to Problem \ref{prob:Eulerian decision}.

\begin{theorem} \label{thrm:3SAT equiv}
Given a 3-SAT instance $I$ and the Eulerian graph with turning costs of $I$,  there is a solution to the 3-SAT instance if and only if there is a zero-cost solution to the corresponding turning cost instance.
% i.e. an Eulerian circuit through the graph with total cost zero.
\end{theorem}

\begin{proof}
Suppose there is a zero-cost solution to the associated turning cost problem. Then, at each vertex $x_i$, only neighbouring pairs of edges can appear consecutively in the Eulerian circuit as these are the only pairs with turning cost zero. Thus, for each $x_i$, because of the parity of the interspersed $x_i u$ edges, there are only two possible configurations for the Eulerian circuit to follow through the vertex $x_i$: one where it joins half-edges of the triangles with positive literal $x_i$'s in their labelling clauses, and one where it joins half-edges of triangles with negative literals  $\neg x_i$ in their labelling clauses, as in Figure~\ref{f.st}.

\begin{figure}
\centering
\hspace{10mm}
\subfigure[$x_i$ is true.]{
\labellist \small\hair 2pt
\pinlabel {$(x_i \! \! \lor \! \!  x_p \! \! \lor \! \!   x_q)$} [r] at   42 92
\pinlabel {$( x_i \! \! \lor \! \!  x_l \! \! \lor \! \!   x_m)$} [l] at   114 92
\pinlabel {\rotatebox{90}{$( \lnot x_i \! \! \lor \! \!  x_n \! \! \lor \! \!   x_r)$}}  at   78 6
\endlabellist
\raisebox{10mm}{\includegraphics[scale=.8]{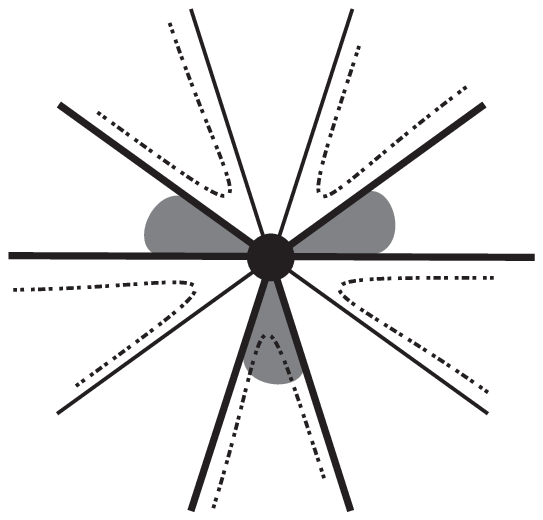}}
 \label{f.st1}
}
\hspace{10mm}
\subfigure[$x_i$ is false.]{
\labellist \small\hair 2pt
\pinlabel {$(x_i \! \! \lor \! \!  x_p \! \! \lor  \! \!  x_q)$} [r] at   42 92
\pinlabel {$( x_i \! \! \lor \! \!  x_l \! \! \lor \! \!   x_m)$} [l] at   114 92
\pinlabel {\rotatebox{90}{$( \lnot x_i \! \! \lor \! \!  x_n \! \! \lor \! \!   x_r)$}}  at   80 6
\endlabellist
\raisebox{10mm}{\includegraphics[scale=.8]{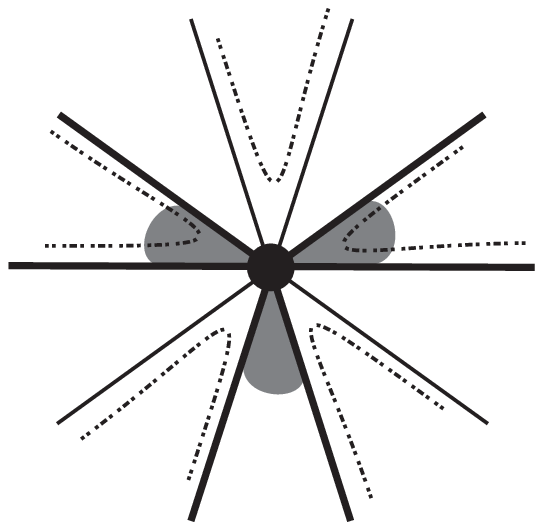}}
 \label{f.st2}
}
\caption{These are the only two possible zero-cost transitions at a vertex $x_i$.  Dotted lines in the left image represent the Eulerian circuit configurations at the vertex $x_i$ corresponding to variable $x_i$ set to true, hence connecting $\lnot x_i$ triangles. The image on the right represents $x_i$ when $x_i$ is set to false.}
\label{f.st}
\end{figure}

In the cases where the triangles with $x_i$ in their labelled clauses are followed, we assign a value of ``false" to $x_i$. In the case that the edges in the $\neg x_i$ labelled triangles are followed, we assign a value of ``true" to $x_i$, again as in Figure \ref{f.st}.

We claim that this is a solution to the given 3-SAT problem. If not, then one of the clauses is false, and we examine the triangle labelled by that clause. There are four cases, depending on how many positive and negative literals are in the clause.  In each case, since the clause is false,  the positive literals must be set to false, and the negative literals set to true.  However, a positive literal set to false corresponds to the Eulerian circuit following the shaded region of a triangle, as does a negative literal set to true.  Thus, in all cases, a disjoint 3-cycle results, and since this cannot occur in an Eulerian circuit, we have contradicted the fact that we have a solution to the associated turning cost problem.

For the converse, we need to show that if there is a solution to the 3-SAT problem, then the associated turning cost problem has a zero-cost solution.  This follows from construction, since we can examine each variable $x_i$ in the 3-SAT expression, and assign the  vertex $x_i$ one of the two transition systems as shown in Figure \ref{f.st} according to whether the variable is assigned a value of true or false in the given solution to the 3-SAT problem.  Both of these transitions systems have cost zero. We then just need to check that these transition systems may always be extended to an Eulerian circuit. This follows since, as in the previous argument, any closed triangle would correspond to an unsatisfied clause, of which there are none.  Thus, each set of edges in any triangle must form a set of paths, each of which continues to edges incident with $u$ (see Figure~\ref{f.pat}).   Since all the paths begin and end at $u$, where every pairing of half-edges has cost zero, we can simply concatenate these paths in any order through $u$ to get a zero-cost solution to the associated turning cost problem.
\hfill$\qed$
\end{proof}

\begin{figure}
\centering
\includegraphics[scale=.6]{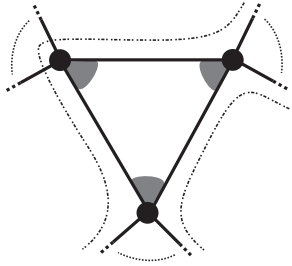}
\caption{Subpaths in an Eulerian circuit around a triangle, recalling that the Euler circuit follows turns with cost zero, i.e. using consecutive edges about a vertex.}
\label{f.pat}
\end{figure}

\begin{corollary}\label{c1}
 Problem~\ref{prob:Eulerian decision}  is NP-complete.
 \end{corollary}
 \begin{proof}
 By Theorem~\ref{thrm:3SAT equiv}, 3-SAT is polynomial time reducible to Problem \ref{prob:Eulerian decision}, and 3-SAT is well-known to be NP-hard. Furthermore, Problem \ref{prob:Eulerian decision} is clearly contained in NP since the turning cost of a given solution may be computed in polynomial time.
\hfill$\qed$
\end{proof}

\begin{corollary}\label{c2}
Problem \ref{prob:Eulerian Costs} is NP-hard.
\end{corollary}

\begin{proof}
If an optimization problem has a corresponding decision problem that is NP-Complete, the optimization problem must be NP-hard.
\hfill$\qed$
\end{proof}

It is convenient at this point make the following technical observation which we will use  in Section~\ref{low}.
\begin{observation}\label{l.deg}
Given a 3-SAT instance $I$, we may assume without loss of generality that any vertex $x_i$ in the $G_I$,   Eulerian graph with turning costs of $I$ that is used in the proof of Theorem~\ref{thrm:3SAT equiv},  has a degree that is not divisible by 4.
\end{observation}
\begin{proof}
If the degree of $x_i$ is divisible by 4, then modify the graph by taking two parallel copies of any edge $(x_i,u)$ from $x_i$ to the apex vertex and making the region between them in the neighborhood of $x_i$ unshaded.  This increases the degree by two, without changing the parity of the number of edges in any unshaded region.  Thus, we are still able to distinguish between the two possible transitions at the vertex $x_i$ (consider Figure~\ref{f.st} with the two additional edges added), so the additional edges do not affect the  ways in which a zero-cost Eulerian circuit can follow the  triangles corresponding to the clauses,  and the construction of the graph is still polynomial time. Thus the proofs of Proposition~\ref{p.polyt}, Theorem~\ref{thrm:3SAT equiv}, and Corollaries~\ref{c1} and~\ref{c2} still hold with the modified graph.
\hfill$\qed$
\end{proof}

\section{Reformulating Problem \ref{prob:Eulerian Costs} as a TSP }
Recall that Traveling Salesman Problem (TSP) seeks a minimum cost Hamilton cycle in a graph with edge weights, that is, a cycle that visits each vertex of the graph with a minimum sum of the edge weights of the edges used in the cycle.
The best-known algorithm for the general version of the TSP is the $O(n^22^n)$ dynamic program described by Held and Karp in 1962 \cite{HK}.
 Because of the practical importance of the TSP, there is a rich history of computational work, including heuristics, integer programming solutions, genetic algorithms, and simulated annealing algorithms (see \cite{LL85} for a survey).
There are good approximations for special cases of the TSP.  For example, in the case of the metric TSP (where costs are non-negative, symmetric, and obey the triangle inequality), there is a simple $3/2$-approximation algorithm, i.e., an algorithm guaranteed to find a solution, in polynomial time, with cost at most $3/2$ times the optimal cost \cite{Chr76}.  Unfortunately finding a good approximation algorithm for the general TSP is (provably) challenging, since such an algorithm would lead to a solution for the NP-hard Hamiltonian cycle problem.

The relevance of all these results for the TSP here is that the problem of finding a minimum turning cost Eulerian circuit  can be reformulated as a TSP. This  means that the  TSP machinery may  be brought to bear on solving Problem~\ref{prob:Eulerian Costs}, and hence on finding optimal threading routes for the scaffolding strand in DNA origami methods of self-assembly. While the cost data is unlikely to be metric in the case of a turning cost problem, in small instances, such as those likely to arise in practice from DNA self-assembly problems, should be tractable using the Held-Karp algorithm \cite{HK} or an integer programming solution.

We will use line graphs to give the connection between Problem \ref{prob:Eulerian Costs} and the TSP.  Let $G$ be a graph, and  $E(G)=\{e_1, \ldots , e_m\}$ be the edges of $G$. The {\em line graph}, $L(G)$, of $G$ has vertex set $V(L(G))=  E(G)$, and edge set $E(L(G))= \{(e_i,e_j) \mid  e_i \text{ and } e_j  \text{ are adjacent edges  in } G \}$.  Clearly,  $L(G)$ may be constructed in polynomial time from $G$.

 We now solve Problem~\ref{prob:Eulerian Costs}, by reducing it to the TSP.

 \begin{theorem} \label{th:TSP solve}
 Given a graph $G$ with turning costs, there is an associated edge weighted graph that may be constructed in polynomial time such that applying the TSP to this graph yields an optimal Eulerian circuit for $G$.
 \end{theorem}

 \begin{proof}
Given an Eulerian graph $G$ equipped with a set of turning costs at each vertex, we first form another graph $G_2$ by subdividing all of the edges of $G$ twice, that is, adding two vertices to each edge of $G$, as in Figure~\ref{f.descb}. Note that even if $G$ has loops or multiple edges, $G_2$ has none. We assign pair costs to $G_2$ as follows. Each edge $e=(u,v)$ of $G$ gives rise to three edges in $G_2$. Call these $e_u, e_c, e_v$, where $e_u$ is incident with  $u$ and $e_v$ is incident with  $v$ (if $u=v$ denote these two edges by $e_{u_1}$ and $e_{u_2}$), and $e_c$ is the remaining edge in the center. Let the turning cost of the pair $(e_v,f_v)$ in $G_2$ equal the turning cost $w_v(e,f)$ in $G$.  Set the turning cost of all other pairs of edges in $G_2$ (these are the edges of the form $(e_v,e_c)$ for some $v$) to zero.  Note that there is a one-to-one correspondence between the Eulerian circuits of $G$ and those of $G_2$, and, moreover, the corresponding Eulerian circuits have the same turning costs.  Clearly $G_2$ with its edge costs may be constructed from $G$ in polynomial, even linear, time.

We now form the line graph, $L(G_2)$, again in polynomial time.  As $G_2$ has no loops or multiple edges, each pair of edges $e_v,f_v$  at $v$ corresponds to a unique edge   $(e_v,f_v)$ in $L(G_2)$.  Assign weights to the edges of $L(G_2)$ as follows.
 If $e$ and $f$ are adjacent edges in $G_2$, then they have an associated turning cost  $w_v(e,f)$.  We give the corresponding edge $(e,f)$ in $L(G_2)$  weight $ w_v(e,f)$.  See Figure~\ref{f.lg}.

 Every walk in $G_2$ is uniquely determined by  its sequence of edges, and, as $L(G_2)$ has no multiple edges, every walk in $L(G_2)$ is determined by a sequence of vertices. If $C$ is an Eulerian circuit in $G_2$ then it is well-known (or see \cite{Ch,HNW}) that $C$ defines a Hamiltonian cycle $H$  in $L(G_2)$.  By construction, the  costs of $C$ in $G_2$, and the weights of $H$ in $L(G_2)$ are equal.
 Conversely, let $H$ be a Hamiltonian cycle in $L(G_2)$.
  Since each vertex  $e_c$ of $L(G_2)$ is of degree 2,  $H$ is of the form  $e^1_{u_1}e^1_{c_1} e^1_{v_1}e^2_{u_2}  \cdots e^m_{c_m} e^m_{v_m}$, where $v_i=u_{i+1}$ and the indices are taken modulo $m$. From this, it is easily seen that the  corresponding walk in $G_2$ is an Eulerian circuit of $G_2$ of cost equal to the weight of $H$. Thus, there is a bijection between Hamiltonian cycles in $L(G_2)$ and  Eulerian circuits in  $G_2$, and hence also the Eulerian circuits in $G$. As this bijection maps costs to weights directly, if $H$ is  a solution to the TSP in $L(G_2)$, then the corresponding Eulerian circuit in $G_2$ gives a minimum turning cost Eulerian circuit of $G_2$ and hence of $G$.
  \hfill$\qed$
  \end{proof}

\begin{figure}
\centering
\subfigure[ $G$.]{
\labellist \small\hair 2pt
\pinlabel {$e$} at   56 144
\pinlabel {$u$} at    33 61
\pinlabel {$v$} at    33 204
\pinlabel {$f$} at   18 228
\pinlabel {$g$} at    56 245
\pinlabel {$h$} at   75  228
\endlabellist
\raisebox{0mm}{\includegraphics[height=5cm]{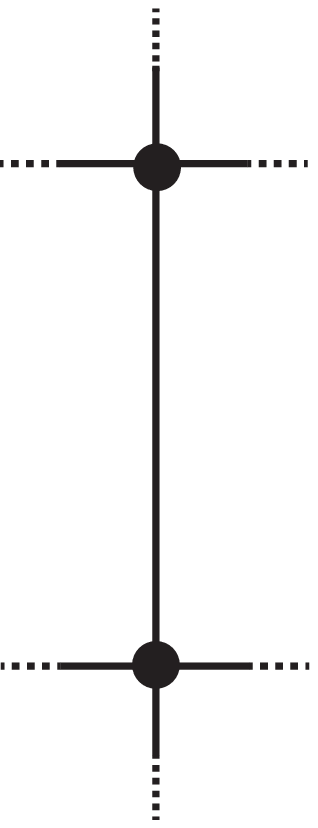}}
 \label{f.desca}
}
\hspace{15mm}
\subfigure[$G_2$, the subdivided graph.]{
\labellist \small\hair 2pt
\pinlabel {$e_c$} at   58 143
\pinlabel {$e_u$} at   58 95
\pinlabel {$e_v$} at   58 185
\pinlabel {$u$} at    33 61
\pinlabel {$v$} at    33 204
\pinlabel {$f_v$} at   18 228
\pinlabel {$g_v$} at    60 247
\pinlabel {$h_v$} at   75  228
\endlabellist
\includegraphics[height=5cm]{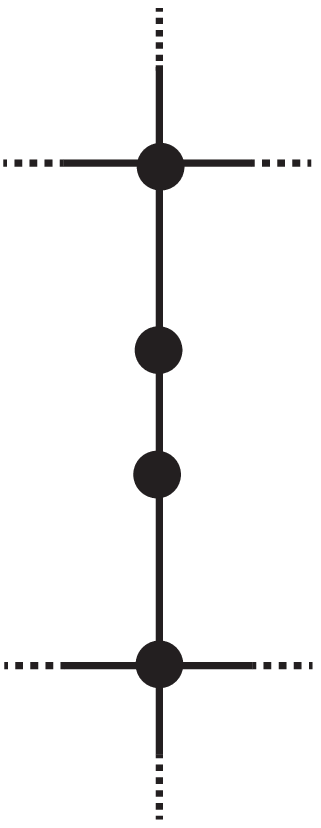}
 \label{f.descb}
}
\hspace{20mm}
\subfigure[$L(G_2)$ with some of its edge weights shown.]{
\labellist \small\hair 2pt
\pinlabel {$w_v(f_v,e_v)$} [r] at    11 176
\pinlabel {$w_v(f_v,g_v)$} [r] at    11 253
\pinlabel {$w_v(g_v,h_v)$} [l] at    142 257
\pinlabel {$w_v(f_v,h_v)$} [l] at     142 238
\pinlabel {$w_v(g_v,e_v)$} [l] at     142 187
\pinlabel {$w_v(h_v,e_v)$} [l] at     142 171
\pinlabel {$0$} [l] at     142 147
\pinlabel {$0$} [l] at     142 110
\endlabellist
\includegraphics[height=5cm]{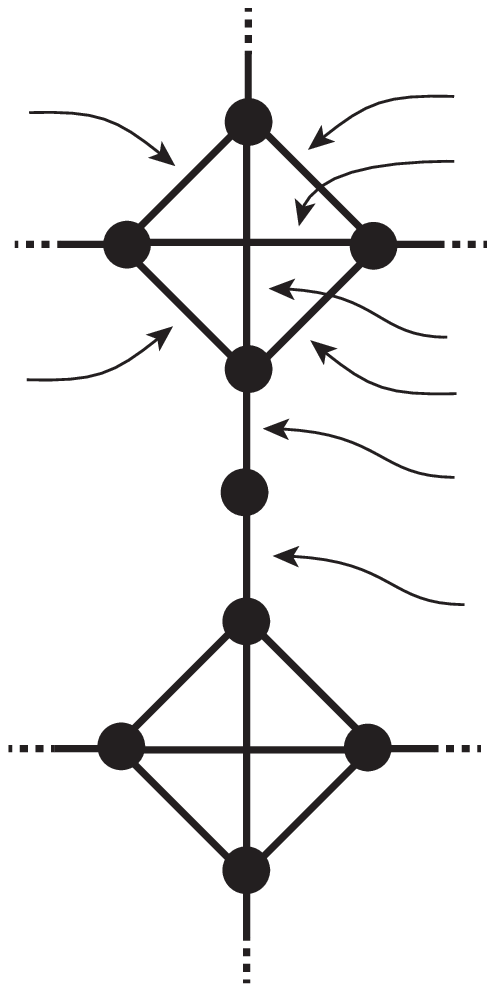}
 \label{f.lg}
}

\caption{A graph $G$,  $G_2$, and $L(G_2)$ with some edge labels and edge-weights shown.}
\label{f.line}
\end{figure}

 Thus we have formulated the problem of finding a minimum turning cost Eulerian circuit, Problem~\ref{prob:Eulerian Costs}, as a TSP. Note that the above argument does not reduce the TSP to Problem~\ref{prob:Eulerian Costs} since not every Hamiltonian graph is a line graph.

\section{Some special cases }

Often special instances of intractable problems may be solved efficiently, and that is true here as well.  We see below that an optimal Eulerian circuit may be found in polynomial time for 4-regular plane graphs with no crossing transitions.  It is fortuitous that this case is tractable, as many likely graph structured targets for DNA origami assembly, for example lattice subsets and cages, are planar, while requiring that a scaffolding strand and staples follow faces without crossing over one another respects the physical constraints of DNA.
On the other hand however, we will also see that the problem remains NP-hard even if we restrict to the class of graphs with maximum degree 8.

\subsection{4-regular plane graphs with no crossing transitions}

If $v$ is a vertex in  a 4-regular plane graph, then it has three transition systems, determined by the embedding in the plane,  as in Figure~\ref{4trans}.  In this section, we will assume the crossing transition systems are prohibited. (This can be done by assigning the pairs that comprise them  large turning costs, in particular, larger than the sum of all the turning costs of other non-crossing transition systems.)

\begin{figure}
\centering
\begin{tabular}{ccccccc}
\labellist \small\hair 2pt
\pinlabel {$v$}  at 36 22
\endlabellist
\includegraphics[scale=.5]{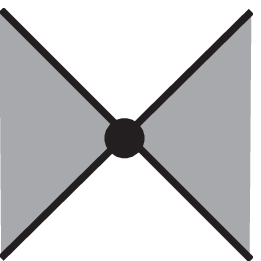}
 & \quad \raisebox{6mm}{$\longrightarrow$} \quad  & \includegraphics[scale=.5]{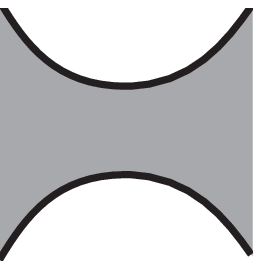} &\hspace{5mm}  & \includegraphics[scale=.5]{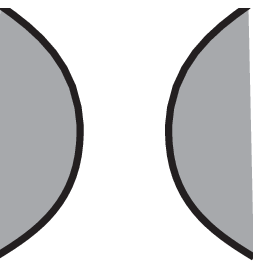} &\hspace{5mm}  &\includegraphics[scale=.5]{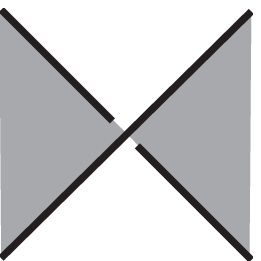} \\
  && white smoothing && black smoothing && crossing.
\end{tabular}
\caption{The three transition systems of a vertex $v$ in a face two colored 4-regular plane graph.}
\label{4trans}
\end{figure}

More generally, if $G$ is an Eulerian graph embedded in some surface, then an {\em A-trail} (or a {\em non-intersecting Eulerian circuit}) of $G$ is an Eulerian circuit in which consecutive edges  in the circuit,  $(v_{i-1},v_i)$ and $(v_i,v_{i+1})$ say, are adjacent in the cyclic ordering of the edges incident to $v_i$. (Thus an Eulerian circuit of a 4-regular plane graph that has no crossing transitions is an A-trail.)
In \cite{Kot66}, Kotzig proved that every 4-regular plane graph contains an A-trail. However,  Bent and Manber, in  \cite{Bent87}, showed that dropping the 4-regularity requirement results in a problem that is NP-complete, i.e.,  the problem of deciding if an Eulerian plane graph contains an A-trail is NP-complete. This remains the case even when restricted to simple, 3-connected graphs with only 3-cycles and 4-cycles as face boundaries (see  \cite{And95}), although  a polynomial-time algorithm for finding A-trails in simple 2-connected outerplane Eulerian graph was given in \cite{And98}.
 Andersen, Bouchet and Jackson \cite{And96} characterised all 4-regular plane graphs that have two orthogonal  A-trails, where two A-trails of $G$ are {\em orthogonal} if the two trails have different transitions  at each vertex of $G$.   Furthermore the complexity of the related problem of finding Eulerian circuits on 4-regular graphs in which only crossing transitions are allowed (which corresponds to finding Eulerian Petrie walks in an underlying graph) has been studied by \v{Z}itnik in \cite{Zit}.

In light of these results, it is in general non-trivial to determine the minimal cost A-trail when turning costs are assigned to the graph. However, we demonstrate below that it can be accomplished in polynomial time for all 4-regular plane graphs.

\begin{theorem}  If $G$ is a 4-regular plane graph with a set of turning costs such that the crossing transitions are prohibited, then an optimal Eulerian circuit may be found in polynomial time.
\end{theorem}

\begin{proof}
We recall that every 4-regular plane graph $G$ is the medial graph of its Tait graph (or blackface graph), as in Figure \ref{taitgraph}  (see, for example, \cite{EMM} for details).  The Tait graph, $F$, is constructed by face 2-colouring $G$ using the colors black and white such that the unbounded region is colored white, and placing a vertex of $F$ in the interior of each black face. (Note that $G$ is face 2-colourable as it is plane and 4-regular.) There is an edge between two vertices in $F$ whenever the two regions corresponding to the vertices have a shared vertex of $G$ on their boundary.  The edge is drawn between the two vertices of $F$, passing through this shared vertex of $G$.  Thus, there is a one-to-one correspondence between the edges of $F$ and the vertices of $G$, and if $v$ is a vertex of $G$, we label the corresponding edges of $F$ by $e_v$.

The face 2-coloring of $G$ allows us to distinguish the two non-crossing transition systems at each vertex as either a black smoothing or a white smoothing, as in Figure \ref{4trans}. (The term smoothing derives from standard terminology in knot theory.) It is well-known that for plane graphs there is a one-to-one correspondence between the spanning trees of $F$ and the Eulerian circuits of $G$ (see e.g. \cite{Las81} or \cite{Ric91}).  The correspondence identifies an edge $e_v$ in a spanning tree of $F$  with a white smoothing at $v$ in the Eulerian circuit of $G$, and an edge $e_u$ not in the spanning tree with a black smoothing at $u$ in the Eulerian circuit.  Again, see Figures \ref{4trans} and \ref{taitgraph}.

Suppose for each vertex $v$ in $G$, the cost for the white smoothing is $a_v$, while the cost for the black smoothing is $b_v$. Then we assign the value $a_v-b_v$ to the edge $e_v$ in $F$.   Now suppose $C$ is an Eulerian circuit without crossing in $G$, and  let $I$ be the vertices of $G$ which have a white smoothing in $C$.  We can see that the total cost of the Eulerian circuit will be,
\[\sum_{v\in I} a_v +\sum_{v\notin I} b_v=\sum_{v\in I} \left(a_v-b_v\right) +\sum_{v\in V} b_v.\]

However, because of the correspondence between Eulerian circuits of $G$ and the spanning trees of $F$, the set of edges $ \{ e_v \mid v \in I\}$ is a spanning tree of $F$.  Since we have assigned the value of $a_v-b_v$ to the edge $e_v$ in $F$, the summand $\sum_{v\in I} \left(a_v-b_v\right)$ on the right-hand side is the weight of this spanning tree.

Thus, a minimum weight spanning tree in $F$ corresponds to a minimum cost Eulerian circuit in $G$.  Since it is well known that minimum weight spanning trees may be found in polynomial time (for example, by Kruskal's algorithm), it follows that optimal Eulerian circuits without crossings may be found for 4-regular plane graphs in polynomial time.
\hfill$\qed$
\end{proof}

\begin{figure}
\centering
\subfigure[A 4-regular plane graph $G$.]{
\quad\includegraphics[scale=.8]{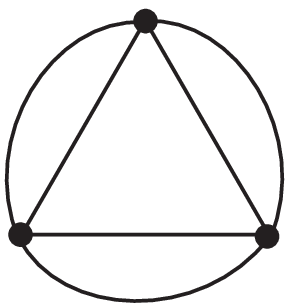}\quad
\label{c1.f6a}
}
\hspace{10mm}
\subfigure[A face 2-colouring of $G$. ]{
\quad\includegraphics[scale=.8]{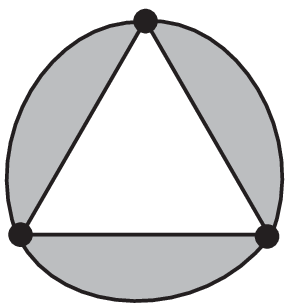}\quad
\label{c1.f6b}
}
\hspace{10mm}
\subfigure[Forming $F$.]{
 \quad\includegraphics[scale=.8]{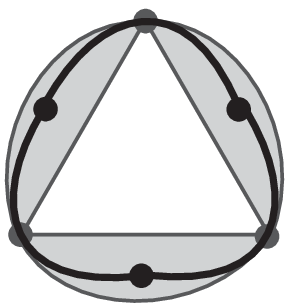}\quad
\label{c1.f6c}
}
\hspace{10mm}
\subfigure[The Tait graph $F$.]{
\quad\raisebox{0mm}{\includegraphics[scale=.8]{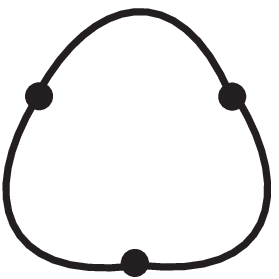}}\quad
\label{c1.f6d}
}
\caption{Forming  Tait  graphs.}
\label{taitgraph}
\end{figure}

\subsection{Graphs of low degree.} \label{low}

The associated graphs $G_I$ used to prove Theorem \ref{thrm:3SAT equiv} may have vertices of very high degree, but we show that restricting ourselves to graphs with low degree vertices, does not, in general, change the complexity of Problem~\ref{prob:Eulerian Costs}.

For the proof of Theorem~\ref{lo} we will need the Cartesian product of graphs. Let $G$ and $H$ be simple graphs (i.e., with no loops or multiple edges). Then the \emph{Cartesian product} $G \Box H$ is the graph with vertex set $V(G)\times V(H)$ and whose edge set is the set of all unordered pairs $((u_1,v_1),(u_2,v_2))$ such that either $(u_1,u_2)\in E(G)$ and $v_1=v_2$, or $(v_1,v_2)\in E(H)$ and $u_1=u_2$.

\begin{figure}
\centering
\subfigure[ A vertex $x_i$ of $G$.]{
\labellist \small\hair 2pt
\pinlabel {$e^i_1$} at   90 232
\pinlabel {$e^i_2$} at   30 185
\pinlabel {$e^i_3$} at     9 120
\pinlabel {$e^i_4$} at    30 49
\pinlabel {$e^i_5$} at   90 1
\pinlabel {$e^i_6$} at     163 1
\pinlabel {$e^i_7$} at     226 49
\pinlabel {$e^i_8$} at    245 120
\pinlabel {$e^i_9$} at     226 185
\pinlabel {$e^i_{10}$} at    163 232
\endlabellist
\raisebox{0mm}{\includegraphics[scale=0.5]{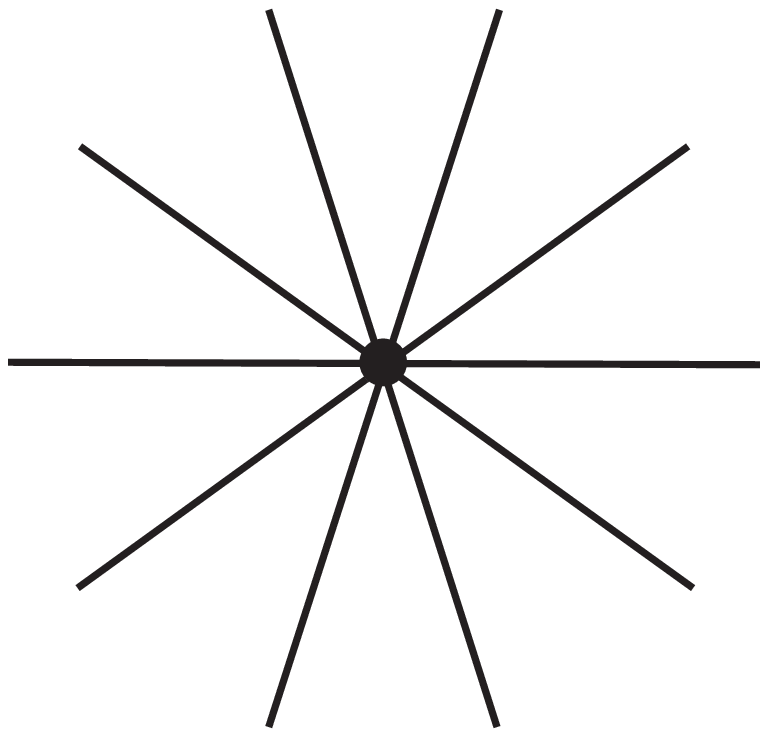}}
 \label{f.b1}
}
\hspace{5mm}
\subfigure[The graph $B_i$.]{
\labellist \small\hair 2pt
\pinlabel {$e^i_1$} at    127 242
\pinlabel {$e^i_2$} at    102 119
\pinlabel {$e^i_3$} at      46 7
\pinlabel {$e^i_4$} at    137 86
\pinlabel {$e^i_5$} at    252 151
\pinlabel {$e^i_6$} at     127 134
\pinlabel {$e^i_7$} at      2 151
\pinlabel {$e^i_8$} at     115 86
\pinlabel {$e^i_9$} at      203 6
\pinlabel {$e^i_{10}$} at    152 119
\endlabellist
\raisebox{0mm}{\includegraphics[scale=0.5]{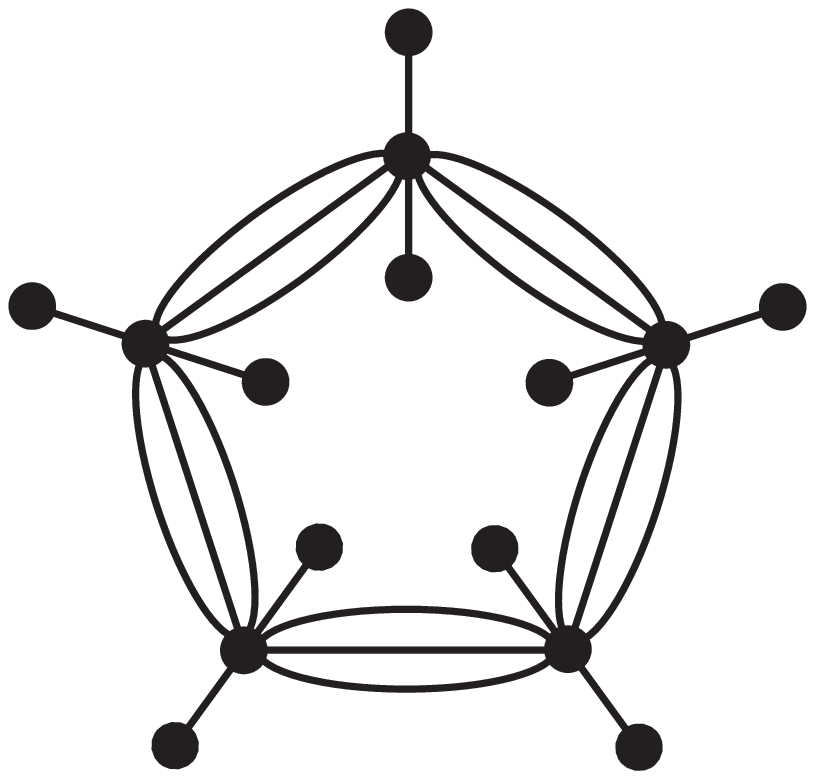}}
 \label{f.b2}
}
\hspace{5mm}
\subfigure[Blowing-up $x_i$.]{
\labellist \small\hair 2pt
\pinlabel {$e^i_1$} at   90 232
\pinlabel {$e^i_2$} at   30 185
\pinlabel {$e^i_3$} at     9 120
\pinlabel {$e^i_4$} at    30 49
\pinlabel {$e^i_5$} at   90 4
\pinlabel {$e^i_6$} at     163 4
\pinlabel {$e^i_7$} at     226 49
\pinlabel {$e^i_8$} at    245 120
\pinlabel {$e^i_9$} at     226 185
\pinlabel {$e^i_{10}$} at    163 232
\endlabellist
\raisebox{0mm}{\includegraphics[scale=0.5]{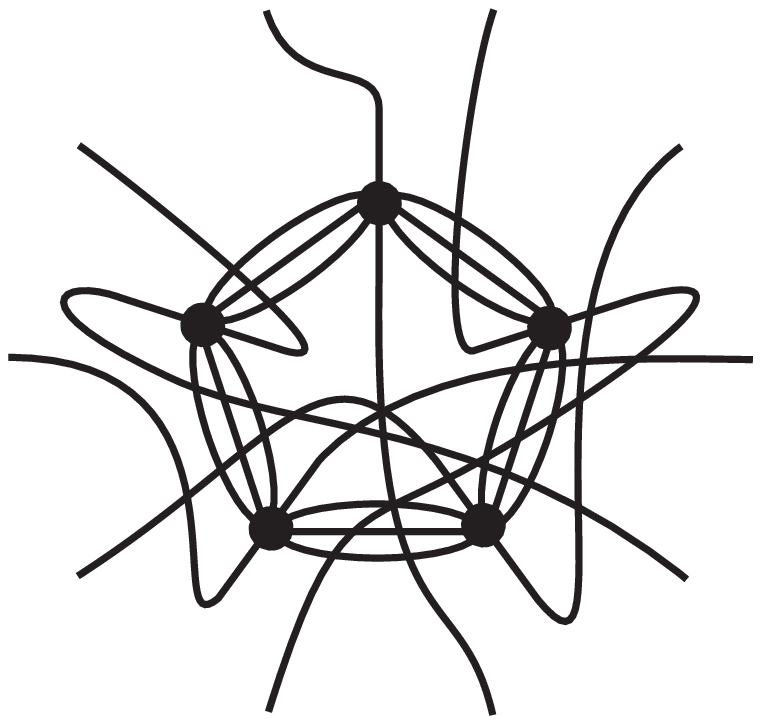}}
 \label{f.b3}
}

\subfigure[The apex vertex $u$.]{
\labellist \small\hair 2pt
\pinlabel {$u$} at   100 106
\pinlabel {$f_1$} at   31 10
\pinlabel {$f_2$} at  59  10
\pinlabel {$f_3$} at     86 10
\pinlabel {$f_4$} at    113 10
\pinlabel {$f_5$} at    140 10
\pinlabel {$f_6$} at  167 10
\endlabellist
\includegraphics[scale=0.5]{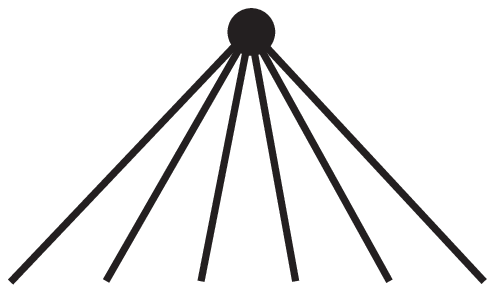}
 \label{f.b4}
}
\hspace{20mm}
\subfigure[The graph $B_u$ for $d=3$.]{
\labellist \small\hair 2pt
\pinlabel {$f_1$} at   31 3
\pinlabel {$f_2$} at  59  3
\pinlabel {$f_3$} at     86 3
\pinlabel {$f_4$} at    113 3
\pinlabel {$f_5$} at    140 3
\pinlabel {$f_6$} at  167 3
\endlabellist
\includegraphics[scale=0.5]{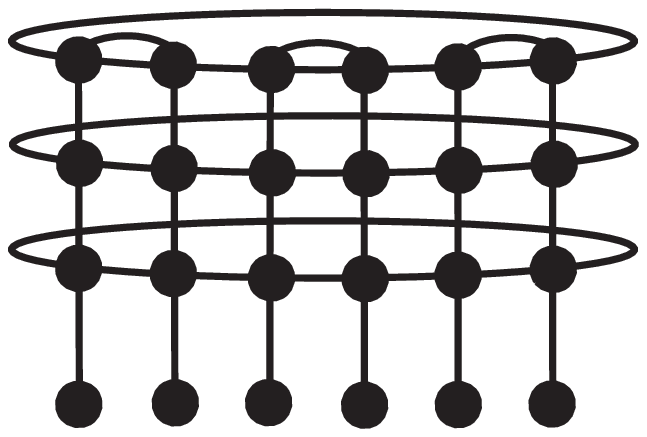}
 \label{f.b5}
}

\caption{The blow-ups $B_i$ and $B_u$.}
\label{f.blows}
\end{figure}

\begin{theorem}\label{lo}
Solving  Problem~\ref{prob:Eulerian Costs} remains NP-hard even restricted to the class of graphs of maximum degree 8.
\end{theorem}

\begin{proof}
Let $G=G_I$ be the Eulerian graph with turning costs  associated with a 3-SAT instance $I$ as constructed in Section~\ref{sec:associated graph}. By  Observation~\ref{l.deg}, we may assume, without loss of generality,  that if a non-apex  vertex of $G$ has degree greater than 8, then its  degree is not divisible by 4. To prove the theorem, we will construct, in polynomial time, a Eulerian graph with turning costs, $G'$, that has maximum degree $\Delta(G')\leq 8$.  Furthermore, $G'$ will have a zero-cost Eulerian circuit if and only if $G$ does.  To construct $G'$,  we `blow-up' each high degree vertex of $G$, replacing it with a special graph that has maximum degree~8. We will need two types of blow-ups: one for the vertices $x_i$ arising from the variables of $I$, and one for the apex vertex $u$.

We denote the  edges incident with  a non-apex vertex $x_i$ of $G$ by $e^i_1, e^i_2, \ldots, e^i_{d(x_i)}$ and we assume that they appear in that cyclic order (with respect the orientation of the plane, as in the construction of $G$). Furthermore,  if $d(u)=2d$, we let $f_1, \ldots , f_{2d}$ denote the edges of $G$ that are  incident with  $u$. See Figures~\ref{f.b1} and~\ref{f.b4}.

For the first type of blow-up, for a vertex $x_i$, we  note that since $d(x_i)$   is even and not divisible by 4, then $d(x_i)/2 = 2k+1$ for some $k$.  We then form a graph $B_i$ as follows.  Start with a plane $(2k+1)$-cycle.  At each vertex, place two nugatory edges (i.e., edges with a degree one vertex) such that one lies in the bounded region, and one in the unbounded region.  Add two parallel copies of each non-nugatory edge to the $(2k+1)$-cycle.   Label the nugatory edges by labelling an arbitrary edge $e^i_1$. Then, if $e^i_j$ has been assigned to some nugatory edge in the unbounded (respectively, bounded) face, travel round the outer face following the orientation of the plane and at the next vertex label the nugatory edge in the bounded (respectively, unbounded) face $e^i_{j+1}$. Continue until all nugatory edges have been labelled, which happens since $2k+1$ is odd.  See Figure~\ref{f.b2}.
Assign pair costs to this graph as follows. Give a cost of zero to any consecutive pairing of half-edges  with respect to the (plane) orientation about each vertex in $B_i$. All other pairings of half-edges have cost one. The resulting graph with turning costs is $B_i$.

To {\em blow-up} a vertex $x_i$ in $G$, we replace it with $B_i$ as follows.   Suppose an edge $e^i_j$  of $G$  has endpoints $x_i$ and $w$.  Then we identify the degree one vertex of  $e^i_j$ in $B_i$ with the vertex $w$.  We do this for each $e^i_j$, and then delete the vertex $x_i$ and its incident edges.

For the blow-up of the apex vertex $u$, which has degree $2d$, we begin by taking the Cartesian product   $C_{2d}\Box P_d $ of a $2d$-cycle and a $d$-path. Consider the $2d$-cycles in $C_{2d}\Box P_d $ that correspond to the original copy of $C_{2d}$. Exactly  two of these cycles contain degree 3 vertices. Denote these two cycles by $C'_{2d}$ and $C''_{2d}$. For each vertex in $C'_{2d}$ attach a nugatory edge, and label these nugatory edges $f_1, \ldots , f_{2d}$. Next take a parallel copy of $d$ distinct non-adjacent edges in the cycle  $C''_{2d}$. Denote the resulting graph by $B_u$. See Figure~\ref{f.b5}. Note that $B_u$ contains exactly $2d$ degree one vertices, which are ends of the edges labelled $f_1, \ldots , f_{2d}$, and every other vertex is of degree 4.  Assign pair costs to the degree 4 vertices of $B_u$ by giving each pairing cost zero.
Observe, for  later, that for every partition of $\{f_1, \ldots , f_{2d}\}$ into pairs, there exists $d$ edge-disjoint paths in $B_u$ such that $f_i$ and $f_j$ are in the same path if and only if they are paired in the partition, in other words, edge-disjoint paths can be found such that each path contains the $f_i, f_j$ from exactly one pair.  One way of doing this is to assign a copy of $C_{2d}$ to each pair $f_i , f_j$.  The path for a pair $f_i , f_j$  begins by following the copy of $P_d$ that $f_i$ is incident to until it intersects the copy of $C_{2d}$ assigned to $f_i , f_j$.  Then follow $C_{2d}$ in either direction until it copy of $P_d$ that $f_j$ is incident to, and follow that $P_d$ to $f_j$.

The vertex $u$ is blown-up similarly to the $x_i$'s by identifying each degree 1 vertex labeled $f_i$ in $B_u$ with the non-$u$ endpoint of the edge labeled $f_i$ in $G$, and then deleting $u$ and its incident edges.

Now, let $G'$ be the graph obtained from $G$ by blowing-up each vertex $x_i$ that has $d(x_i)>8$ using $B_i$, and, if $d(u)>8$, blowing-up $u$ using $B_u$. The turning costs of $G'$ are inherited from those of the $B_i$, $B_u$, and those of the vertices of $G$ with degree at most 8.  Observe that $G'$ is constructed from $G$ in polynomial time in the number of edges and vertices.
It remains to show that $G$ has a zero-cost Eulerian circuit if and only if $G'$ does.

Suppose that $C'= g_1g_2 \cdots g_p$ is a zero-cost Eulerian circuit of $G'$, specified by the edges $g_i$ of $G'$. (We read all  circuits here cyclically so that we regard, for example, $g_pg_1$ as a subtrail of $C'$.)
Consider a subtrail in $C'$ of the form $e^i_j w e^i_k $, where $w$ is a trail contained entirely in $B_i$. By examining Figure~\ref{f.b2}, observe that  $k=j+1$ or $k=j-1$  (otherwise there is a cost  greater than zero, or $C'$ contains more than one closed walk and hence is not an Eulerian circuit). Thus any subtrail in $C'$ contained in any $B_i$ is of the form $e^i_j w e^i_{j\pm 1}$, and we can obtain a zero-cost pairing at the vertex $x_i$ of $G$ by deleting the subtrail $w$.
Now given $C'$, construct a subsequence $C$ by reading through $C'$ (cyclically).  Whenever there is a subtrail $e^i_j w e^i_{j\pm 1}$ with $w$ contained entirely in $B_i$, delete $w$; whenever there is a subtrail of the form $f_iwf_j$ with  $w$ contained entirely in $B_u$, delete $w$. Then $C$ defines an Eulerian circuit in $G$ (since the edges contained in $C$ are exactly the edges of $G$; and since each  $e^i_j e^i_{j\pm 1}$, and each $f_i f_j$ define valid pairings in $G$). Moreover, $C$ is of cost zero since each pairing $e^i_j e^i_{j\pm 1}$, and $f_i f_j$ in $G$ is of cost zero. Thus if there is a zero-cost Eulerian circuit in $G'$, there is  a zero-cost  Eulerian circuit in $G$.

Conversely, let  $C= g_1w_1g_2w_2 \cdots g_pw_1$ be a zero-cost Eulerian circuit of $G$, where the $g_i$ are edges and the $w_i$ are vertices. Any subtrail  $e^i_j x_i e^i_{j\pm 1}$ in $C$  determines a unique zero-cost trail $e^i_j w_{ij} e^i_{j\pm 1}$ through $B_i$ (see Figures~\ref{f.b1}--\ref{f.b3}). Also, $C$ determines a partition of $\{f_1, \ldots , f_{2d}\}$ into pairs where the pairs correspond to the subtrails $f_iuf_j$. As observed above, there exist $d(u)/2$ disjoint paths in $B_u$ such that $f_i$ and $f_j$ are in the same path if and only if they are paired in the partition. For such a set of disjoint paths, let $P_{ij}$ denote the one that contains $f_i$ and $f_j$.
Now read through $C$ (cyclically),  replace each pair $e^i_j e^i_{j\pm 1}$ with $e^i_j w_{ij} e^i_{j\pm 1}$, and replace each pair $f_iuf_j$ with $P_{ij}$. Denote the resulting sequence by $\tilde{C}'$. We have that $\tilde{C}'$ determines a circuit in $G'$. Moreover, all turning costs in this circuit are zero. However,    $\tilde{C}'$ may not be an Eulerian circuit. This is since there may be edges in $B_u$ that are not in  $\tilde{C}'$. By construction, any edges of $G'$ that are not in $\tilde{C}'$ must be unlabelled edges of $B_u$. We extend $\tilde{C}'$ to include these edges  as follows. As $G'$ is Eulerian,  each component of $G'\backslash E(\tilde{C}')$ is even, and therefore contains an Eulerian circuit. Denote these Eulerian circuits by $D_1, \ldots, D_q$. Extend $\tilde{C}'$ to an Eulerian circuit $C'$ of $G'$ by following  $\tilde{C}'$ until we meet a vertex in one of the $D_j$, detour round this circuit and remove it from the list of circuits, then continue along $\tilde{C}'$, repeating this process until each $D_j$ has been used. As $G$ is connected, this results in an Eulerian circuit. Moreover, as all transitions in $B_u$ have zero-cost, $C'$ is a zero-cost Eulerian Circuit.  Thus if there is a zero-cost Eulerian circuit in $G$ there is a   to a zero-cost  Eulerian circuit in $G$, completing the proof of the theorem
\hfill$\qed$
\end{proof}

\section{Further implications}

Our results have  immediate ramifications for biomolecular computing.  In addition there are other closely related problems in the literature, for example mill routing, about which the results here also inform.

Our primary motivation for this investigation was finding design strategies for self-assembling structures (see \cite{thetilespaper,website}, for example).  One of the measures of the goodness of a design strategy, particularly one that is to be used as part of a biomolecular computing process, is how efficiently it may be found.  Finding an optimal threading for the scaffolding strand of a DNA origami construction of a graph theoretical structure corresponds to finding an optimal Eulerian circuit with turning costs, and we have shown here that this problem is intractable. This is an important first step in determining the complexity of the \emph{input} to biomolecular computing problems. A consequence is that single strand DNA origami methods may not be suitable as a generic starting point for efficient biomolecular computation of graph invariants.  While there are some provably optimal design strategies for other construction methods,  for example for branched junction molecules methods of DNA self-assembly of a few common classes of graphs (see \cite{thetilespaper,website}), the computational complexity of the general problem  for these other methods has not yet been studied.

While we have addressed here a computational question arising from origami folding, an earlier assembly method of DNA self-assembly uses a design strategy of tracing each edge of the graph twice, once in each direction, while prohibiting double-backs (hairpin turns) and other disconnections at the vertices (see \cite{NS02,NSW09}).  To our knowledge, the computational complexity of this method has not yet been addressed beyond its correspondence with graph genus.  While it is possible that the techniques presented here could be adapted to this earlier method, it is not immediate, since simply doubling the edges of the graph and forbidding only double-backs gives a very special case of the turning cost problem, and might conceivably be tractable. Furthermore, simply doubling the edges does not guarantee a solution where each original edge is traversed in opposite directions.  Also, finding a route for the scaffolding strand and placing the staples, then removing the nicks between the staples, while it does cover each edge twice in opposite directions, generally results in covering the graph with multiple circular strands, not just one.  Nonetheless, the methods presented here offer a possible approach to analysing open questions arising from this earlier assembly method.

Mill routing is another problem very closely related to Problem \ref{prob:Eulerian Costs}.  In discrete thin mill routing, a router needs to cover all the edges of a graph, which, in the case of orthogonal discrete thin mill routing, is a subset of a grid (see \cite{A+05}).   It is fastest for a router to go straight across when encountering a grid point,  more time consuming for it to turn left or right, and quite slow for it to go back the way it came.  Thus, this is a turning cost problem.

While  Problem \ref{prob:Eulerian Costs} is closely related to the mill routing problem, it differs in that the mill routing problem allows edges to be repeated, and even to  ``double back'' on an edge, while we prohibit this as it should be avoided with DNA origami.   Such subtle differences can have a profound effect on the computational complexity of a problem, as can be seen for example among the variations of the Chinese Postman problem, where the original problem may be solved in polynomial time, but even minor modifications such as directing some subset of the edges lead to NP-hardness (see \cite{EGL95} and \cite{Fle90,Fle91} for overviews).

The mill routing problem was shown in \cite{A+05} to be NP-hard in the general case of any underlying graph, but again allowing edges to be traversed more than once.  Thus, the results of \cite{A+05}, which show that the mill routing problem is NP-hard, do not apply here.  The results here, however, show that the discrete thin mill routing problem remains NP-hard, even in the special case that the desired tour must be an Eulerian circuit.  Thus, even if  a set of repeated edges is specified  ahead of time (this corresponds to doubling an edge), or if, in general, augmenting edges are added to make the graph Eulerian, then the general  discrete thin mill routing problem remains intractable if the desired tour must be an Eulerian circuit.

\section*{Acknowledgement} We thank Ned Seeman for specific design problems leading to this research and for many related discussions.

\end{document}